\documentclass[a4paper,12pt,reqno]{amsart}

\usepackage[normalem]{ulem}

\usepackage{fullpage}
\usepackage{amssymb}
\usepackage{amsmath}
\usepackage{mathtools}
\usepackage{enumitem}
\usepackage{mathrsfs}
\usepackage[all]{xy}
\usepackage{mathtools}
\usepackage{hyperref}
\usepackage{url}
\usepackage{bbm}
\usepackage{longtable}
\usepackage{dsfont}
\usepackage{upgreek}
\usepackage{lmodern}
\usepackage[OT2, T1]{fontenc}

\DeclareSymbolFont{cyrletters}{OT2}{wncyr}{m}{n}
\usepackage[british]{babel}

\usepackage[margin=0.75in]{geometry}
\numberwithin{equation}{section} \numberwithin{figure}{section}

\DeclareSymbolFont{cyrletters}{OT2}{wncyr}{m}{n}
\DeclareMathSymbol{\Sha}{\mathalpha}{cyrletters}{"58}
\DeclareMathSymbol{\Be}{\mathalpha}{cyrletters}{"42}

\newcommand\F{\mathbb{F}}
\renewcommand\P{\mathbb{P}}
\newcommand\Z{\mathbb{Z}}
\newcommand\N{\mathbb{N}}
\newcommand\Q{\mathbb{Q}}
\newcommand\R{\mathbb{R}}

\renewcommand{\b}{\mathbf}

\renewcommand{\gcd}{\textrm{gcd}} 
\renewcommand{\leq}{\leqslant}
\renewcommand{\geq}{\geqslant}
\renewcommand{\#}{\sharp}

\newtheorem{lemma}{Lemma}

\newtheorem{theorem}[lemma]{Theorem}
\newtheorem{proposition}[lemma]{Proposition}
\newtheorem{corollary}[lemma]{Corollary}

\theoremstyle{definition}

\newtheorem{remark}[lemma]{Remark}

\usepackage[usenames,dvipsnames]{color}

\numberwithin{lemma}{section}

\title{Local solubility of a family of ternary conics over a biprojective base I}
\author{Cameron Wilson}
\date{\today}
\address{Department of Mathematics \\
University of Glasgow \\ G12~8QQ United Kingdom}
\email{c.wilson.6@research.gla.ac.uk}

\subjclass[2020] {
11G99, 
11N36, 
\bf{14G05}} 
\date{\today}

\begin{document}

\begin{abstract}
    Let $f,g\in\Z[u_1,u_2]$ be binary quadratic forms. We provide upper bounds for the number of rational points $(u,v)\in\P^1(\Q)\times\P^1(\Q)$ such that the ternary conic
    \[
    X_{(u,v)}: f(u_1,u_2)x^2 + g(v_1,v_2)y^2 = z^2
    \]
    has a rational point. We also give some conditions under which lower bounds exist.
\end{abstract}

\maketitle

\tableofcontents

\section{Introduction}
Let $f,g\in\Z[u_1,u_2]$ be binary quadratic forms which have non-zero discriminant. This paper is concerned with the local solubility of the ternary conics
\begin{equation}\label{ternaryfibres}
    X_{(u,v)}:f(u_1,u_2)x^2+g(v_1,v_2)y^2=z^2
\end{equation}
for $(u,v)=([u_1:u_2],[v_1:v_2])\in\P^1(\Q)\times\P^1(\Q)$, where $(u_1,u_2),(v_1,v_2)\in\Z_{prim}^2$. In particular, we aim to provide upper and lower bounds for the counting problem
\begin{equation}\label{ac-count}
    N_{f,g}(B) = \sharp\{(u,v)\in\P^1(\Q)\times\P^1(\Q): X_{(u,v)}(\Q)\neq\emptyset, \; H((u,v))\leq B\}
\end{equation}
where here and throughout, $H((u,v)) = \max\{|u_1|,|u_2|\}^2\cdot\max\{|v_1|,|v_2|\}^2$ denotes the anticanonical height on $\P^1(\Q)\times\P^1(\Q)$. Formally we view the ternary conics $X_{(u,v)}$ as the rational fibres of the map $\pi:X\rightarrow \P^1\times\P^1$ where
\[
X : f(u_1,u_2)x^2+g(v_1,v_2)y^2=z^2 \subseteq \P^2\times(\P^1\times\P^1)
\]
and $\pi((x,(u,v)))=(u,v)$ for all $(x,(u,v))\in\P^2\times(\P^1\times\P^1)$. Our main result is the following.

\begin{theorem}\label{acheightbound}
    Let $B\geq 2$ and let $f,g\in\Z[u_1,u_2]$ be binary quadratic forms with non-zero discriminant. Then
    \begin{align*}
    N_{f,g}(B) \ll \begin{cases}
    B \;&\text{if at least one of $f$ and $g$ split over $\Q$},\\
    B\log\log B\;&\text{if neither $f$ nor $g$ split over $\Q$}.
    \end{cases}
    \end{align*}
    where the implied constant depends at most on $f$ and $g$.
\end{theorem}

We will see in \S \ref{Discussion} that these upper bounds are larger than what is expected for this problem. It is natural to ask whether these bounds are sharp. In \cite{BLS} an abundance of everywhere locally soluble fibres was shown by demonstrating the presence of a thin set. We will also examine the following thin sets which occur in our problem:
\[
N_{\textrm{thin},1} = \{(u,v)\in\P^1(\Q)\times\P^1(\Q):f(u_1,u_2)=\square\;\text{or}\;g(v_1,v_2)=\square\},
\]
and
\[
N_{\textrm{thin},2} = \{(u,v)\in\P^1(\Q)\times\P^1(\Q): f(u_1,u_2)g(v_1,v_2)=\square\}.
\]
The fibres of points inside each of these sets have a rational point which is easily identified. Our second theorem studies the arithmetic of these thin sets with respect to the anticanonical height. To this purpose, define
\[
N_{\textrm{thin},i}(B) = \#\{(u,v)\in N_{\textrm{thin},i}: H((u,v))\leq B\}.
\]

\begin{theorem}\label{thinsetbound}
We have the following bounds,
\begin{enumerate}
\item $N_{\textrm{thin},1}(B) = 0$ if neither $f$ nor $g$ represents a square;
\item $B \ll N_{\textrm{thin},1}(B) \ll B$ if at least one of $f$ or $g$ represents a square;
\item $N_{\textrm{thin},2}(B)\ll B^{2/3}(\log B)^9$. 
\end{enumerate}
\end{theorem}

We therefore see that the presence of $N_{\textrm{thin},1}$ provides a sharp lower bound for $N_{f,g}(B)$ whenever at least one of $f$ or $g$ split over $\Q$ and at least one of them represents a square number.

\begin{remark}\label{partIIteaser}
As in \cite{BLS} we may ask what happens outside of these thin sets. This amounts to studying the quantity
\[
N^*_{f,g}(B) = \sharp\left\{(u,v)\in\P^1(\Q)\times\P^1(\Q):\begin{array}{l} f(u_1,u_2)\neq\square,\;g(v_1,v_2)\neq\square,\\ X_{(u,v)}(\Q)\neq\emptyset, \; H((u,v))\leq B \end{array}\right\}.
\]
We will study this problem in subsequent work using character sums.
\end{remark}

Problems of this type go back to Serre \cite{SerreConics} who gave upper bounds for the number of triples $(a,b,c)\in\Z^3$ such that $\gcd(a,b,c)=1$, $\max\{|a|,|b|,|c|\}\leq B$ and such that the conic $$ax^2+by^2+cz^2=0$$ is everywhere locally soluble. In this case the family of ternary diagonal conics is parameterised by $\P^2(\Q)$. Upper bounds for more general families of varieties parameterised by $\P^n(\Q)$ have followed, for example in the work of Loughran and Smeets \cite{LS} and Browning, Lyczak and Smeets \cite{BLSmeets}. A more extensive exposition of results is presented in \cite{LRS}, which also gives asymptotics for the case of ternary diagonal conics.

Recently cases where the base of fibration is not projective space have been studied, \cite{BLS,L,LTbT}. In \cite{BLS} it was discovered that unexpected phenomena may occur when the base is biprojective, namely that the existence of thin sets may give more everywhere locally soluble fibres than expected. Outside of these thin sets, even more exotic behaviour was discovered in recent work of the author \cite{Me}. The goal of this work is to study more families over $\P^1(\Q)\times\P^1(\Q)$ so that this unusual behaviour may be considered in a broader context.

\subsection{Acknowledgements} The author would like to thank Tim Browning for suggesting this problem and for feedback on an earlier draft to the paper. Much of the groundwork for this project was completed during the author's visit to ISTA during April 2024 - the author would like to thank their hospitality. Thanks are also due to Efthymios Sofos for helpful discussions and comments. Funding was provided by a Ph.D Scholarship awarded by the Carnegie Trust for the Universities of Scotland.

\section{Expectation VS Reality}\label{Discussion}
To begin we make some preliminary remarks about how this problem relates to the Serre problem and the conjecture of Loughran and Smeets \cite{LS}. First we note that $X$ is non-singular - indeed, the Jacobian of $X$ is given by
\[
\left(\frac{\partial f(u_1,u_2)}{\partial u_1}x^2,\frac{\partial f(u_1,u_2)}{\partial u_2}x^2,\frac{\partial g(v_1,v_2)}{\partial v_1}y^2, \frac{\partial g(v_1,v_2)}{\partial v_2}y^2,2xf(u_1,u_2),2yg(v_1,v_2),-2z\right).
\]
In order for this to be zero, we must have $z=0$, and so at most one of $x$ and $y$ can be zero since $[x:y:z]\in\P^2$. Suppose $x\neq 0$, then we require that
\[
\frac{\partial f(u_1,u_2)}{\partial u_1} = \frac{\partial f(u_1,u_2)}{\partial u_2} = f(u_1,u_2) = 0,
\]
but this is not possible as we require that $f$ has non-zero discriminant and is thus non-singular. Symmetrically, if $y\neq 0$ we cannot have
\[
\frac{\partial g(v_1,v_2)}{\partial v_1} = \frac{\partial g(v_1,v_2)}{\partial v_2} = g(v_1,v_2) = 0.
\]
This is in contrast to the problem considered in \cite{BLS,Me} which deals with a singular variety $X$ (though this may be dealt with by passing to a desingularisation). Next suppose
\[
\widetilde{X} : f(a_0,a_2)x^2+g(a_0,a_3)y^2=z^2,\; a_0a_1=a_2a_3 \subseteq \P^2\times\P^3.
\]
and
\[
Y : a_0a_1=a_2a_3 \subseteq \P^3
\]
Counting with respect to the anticanonical height on $\P^1(\Q)\times\P^1(\Q)$ corresponds to counting soluble fibres of the map $\phi:\widetilde{X}\rightarrow Y$, i.e. counting $y\in Y(\Q)$ such that $\phi^{-1}(y)(\Q)\neq \emptyset$ and $\widetilde{H}(y)\leq \sqrt{B}$, where $\widetilde{H}$ is the naive height on $\P^3(\Q)$. This can be seen by applying the parameterisation
\[
a_0=u_1v_1,\;a_1=u_2v_2,\;a_2=u_2v_1,\;a_3=u_1v_2
\]
of $Y$ by $\P^1\times\P^1$. This relates the current work to \cite{BLS,Me} which are presented as local solubility problems with base of fibration $Y$.

We now discuss the expected behaviour of $N(B)$. To do so we will consider a more general setting. Suppose $\phi:Z\rightarrow W$ is a dominant map between smooth projective varieties $Z$ and $W$ over $\Q$ with geometrically integral fibres admitting multiple fibres. We define $\Delta(\phi)$ as in \cite[equation $(1.3)$]{BL}. Set $W^{(1)}$ to be the collection of codimension $1$ points of $W$. For any $D\in W^{(1)}$, the absolute Galois group $\textrm{Gal}(\overline{\kappa(D)}/\kappa(D))$ of the residue field of $D$ acts on the irreducible components of the reduced fibres $\pi^{-1}(D)\otimes\overline{\kappa(D)}$. Choose some finite subgroup $\Gamma_D(\phi)$ through which the action is factored and define $\Gamma_D^{\circ}(\phi)$ to be the collection of $\gamma\in\Gamma_D(\phi)$ which fix some multiplicity $1$ irreducible component of $\pi^{-1}(D)\otimes\overline{\kappa(D)}$. Then define $\delta_D(\phi) = \#\Gamma_D^{\circ}(\phi)/\#\Gamma_D(\phi)$ and
\[
\Delta(\phi) = \sum_{D\in W^{(1)}}(1-\delta_D(\phi)).
\]
The assumption that the generic fibre is geometrically integral ensures that this sum is finite. Assume further that $W$ does not contain any accumulating thin sets for rational points (see \cite{Peyre} for a definition). This means that $W$ satisfies the conditions of the version of Manin's conjecture found in \cite{ManinandCo.}, which predicts that
\[
\#\{w\in W(\Q): \overline{H}(w)\leq B\} \sim a_{W}B(\log B)^{\rho_{W}-1}
\]
where $a_{W}$ is constant, $\rho_W$ is the Picard rank of $W$ and $\overline{H}$ is the anticanonical height for $W$. Then, following \cite{BLS}, we may consider the prediction that
\begin{equation}\label{generalbaseprediction}
    \#\left\{w\in W(\Q) : \begin{array}{
    c} \phi^{-1}(w)\;\text{is everywhere locally soluble}\\
         \overline{H}(w)\leq B.
    \end{array}\right\} \sim \frac{a'_W B(\log B)^{\rho_{W}-1}}{(\log B)^{\Delta(\phi)}}
\end{equation}
for some constant $a'_W>0$. In our case we have
\[
\#\{(u,v)\in\P^1(\Q)\times\P^1(\Q): H((u,v))\leq B\} \sim cB(\log B)
\]
for some constant $c>0$. The following result, proven in \S \ref{Deltaproof}, characterises how the value of $\Delta(\pi)$ depends on the splitting behaviour of $f$ and $g$:

\begin{proposition}\label{DELTA}
    Let $f,g\in\Z[u_1,u_2]$ be binary quadratic forms with non-zero discriminant and $X_{(u,v)}$ be as in \eqref{ternaryfibres}. Then
    \begin{align}
    \Delta(\pi) = \begin{cases}
        2\;\;\text{if both $f$ and $g$ split over $\Q$},\\
        \frac{3}{2}\;\;\text{if only one of $f$ and $g$ split over $\Q$},\\
        1\;\;\text{if neither $f$ nor $g$ split over $\Q$}.
    \end{cases}
    \end{align}
\end{proposition}

By following \eqref{generalbaseprediction} we therefore find that the predicted orders of growth are as follows:

\[
N_{f,g}(B) \sim \begin{cases}
    \frac{c_0 B}{\log B}\;\;\text{if both $f$ and $g$ split over $\Q$},\\
        \frac{c_1 B}{\sqrt{\log B}}\;\;\text{if only one of $f$ and $g$ split over $\Q$},\\
        c_2 B\;\;\text{if neither $f$ nor $g$ split over $\Q$}.
\end{cases}
\]

Theorems \ref{acheightbound} and \ref{thinsetbound} together give counter examples to prediction \eqref{generalbaseprediction} whenever at least one of $f$ or $g$ splits over $\Q$ and at least one of them represents a square. It is notable that the thin sets considered in \cite{BLS} and this paper are not accumulating with respect to Manin's conjecture. We may further predict that \eqref{generalbaseprediction} should hold after the removal of any corresponding accumulating thin set. In \cite{Me} this is proven to be false in the case considered by \cite{BLS} - in fact an unexpected extra growth factor of $\log\log B$ was discovered. We will study what occurs outside of $N_{\textrm{thin},1}$ and $N_{\textrm{thin},2}$ in later work, as suggested by Remark \ref{partIIteaser}. Note, however, that these behaviours are consistent with Conjecture $3.15$ of \cite{LRS} as their assumption $(1.3)$ excludes the cases when one of $f$ or $g$ split. 


\section{Proof of Proposition \ref{DELTA}}\label{Deltaproof}
First note that if $U=[U_1:U_2]\in\P^1$ is a point on $f(u_1,u_2)=0$, then $D_U: u=U$ is a co-dimension one point of $\P^1\times\P^1$ with generic fibre
\[
g(v_1,v_2)y^2 = z^2,
\]
Since $g$ does not factorise as the square of a linear factor, this fibre is split precisely when it is reducible over some quadratic number field $K$, splitting as two lines. Furthermore, the action of the nontrivial element of the Galois group $\mathrm{Gal}(K/\Q)$ permutes the two lines. It follows that $\delta_{D_U} = \frac{1}{2}$ for each such codimension $1$ point. Symmetrically, for any $V=[V_1:V_2]\in\P^1$ such that $g(V_1,V_2)=0$.\\

Second, recall that if $f$ is reducible over $\Q$, then $f(u_1,u_2)=0$ has two scheme theoretic points over $\P^1$, each corresponding to a separate codimension $1$ point of $\P^1\times\P^1$ with reducible fibre under $\pi$. On the other hand, if $f$ is instead only reducible over a quadratic number field $K/\Q$, then the two points of $f(u_1,u_2)=0$ in $\P^1(K)$ are Galois conjugates and thus only contribute a single scheme theoretic point. Thus we only obtain a single codimension $1$ point of $\P^1\times\P^1$ with reducible fibre.\\

From these two points we conclude the following:
\begin{itemize}
    \item If both $f$ and $g$ split over $\Q$ we obtain $4$ codimension $1$ points, $D_1,\ldots,D_4$ say, of $\P^1\times\P^1$, on which all reducible fibres lie. For each $1\leq i\leq 4$, $\delta_{D_i}=\frac{1}{2}$. Thus $\Delta(\pi) = 2$.
    \item If only $f$ splits over $\Q$ then its points in $\P^1$ correspond to $2$ codimension $1$ points of $\P^1\times\P^1$, $D_1$ and $D_2$ say. On the other hand, $g$ splits over a quadratic extension and so gives rise to a single codimension $1$ point with reducible fibres, say $D_3$. Since $\delta_{D_i} =\frac{1}{2}$ for each $1\leq i\leq 3$, $\Delta(\pi) =\frac{3}{2}$. Similarly if $g$ splits and $f$ does not.
    \item If neither $f$ nor $g$ splits over $\Q$, then we will only obtain two codimension $1$ points of $\P^1\times\P^1$ with reducible fibres, $D_1$ and $D_2$ say. Since $\delta_{D_i} =\frac{1}{2}$ for each $1\leq i\leq 2$, and so $\Delta(\pi) = 1$.
\end{itemize}

\section{Local densities}\label{Localdensities}
In this section we compute the densities of locally soluble fibres. To this purpose we define the set $\Omega'_p\subseteq(\Z/p^2\Z)^4$ for each prime $p$. If $p|2\mathrm{disc}(f)\mathrm{disc}(g)$, then set $\Omega'_p=\emptyset$; otherwise, set
\[
\Omega'_p = \Omega_{p,f} \sqcup \Omega_{p,f,g} \sqcup \Omega_{p,g}
\]
where
\[
\Omega_{p,f} = \left\{(u_1,u_2,v_1,v_2)\in(\Z/p^2\Z)^4: \begin{array}{l} f(u_1,u_2)\equiv 0\bmod{p},\;f(u_1,u_2)\not\equiv0\bmod{p^2},\\ g(v_1,v_2)\not\equiv\square\bmod{p},\;(u_1,u_2),(v_1,v_2)\not\equiv(0,0)\bmod{p}\end{array}\right\},
\]

\[
\Omega_{p,g} = \left\{(u_1,u_2,v_1,v_2)\in(\Z/p^2\Z)^4: \begin{array}{l} g(v_1,v_2)\equiv 0\bmod{p},\;g(v_1,v_2)\not\equiv 0 \bmod{p^2},\\ f(u_1,u_2)\not\equiv\square\bmod{p},\;(u_1,u_2),(v_1,v_2)\not\equiv(0,0)\bmod{p}\end{array}\right\},
\]
and
\[
\Omega_{p,f,g} = \left\{(u_1,u_2,v_1,v_2)\in(\Z/p^2\Z)^4: \begin{array}{l} f(u_1,u_2)\equiv g(v_1,v_2)\equiv 0\bmod{p},\\f(u_1,u_2),g(v_1,v_2)\not\equiv 0 \bmod{p^2},\\ -f(u_1,u_2)g(v_1,v_2)/p^2\not\equiv\square\bmod{p},\\(u_1,u_2),(v_1,v_2)\not\equiv(0,0)\bmod{p}\end{array}\right\}.
\]
The following lemma tells us that these residue sets give insoluble fibres.

\begin{lemma}\label{badresidueclassesarebad}
    Let $p\nmid 2\mathrm{disc}(f)\mathrm{disc}(g)$ be a prime. If $((u_1,u_2),(v_1,v_2))\in\Z_{prim}^2\times\Z_{prim}^2$ reduces to some $(\bar{u}_1,\bar{u}_2,\bar{v}_1,\bar{v}_2)\in \Omega'_{p}$ modulo $p^2$ then $X_{(u,v)}(\Q_p)=\emptyset$.
\end{lemma}

\begin{proof}
    For $(\b{u},\b{v})=((u_1,u_2),(v_1,v_2))\in\Z_{prim}^2\times\Z_{prim}^2$, let $F=f(u_1,u_2)$ and $G=g(v_1,v_2)$. Suppose $(\b{u},\b{v})$ reduces to $(\bar{u}_1,\bar{u}_2,\bar{v}_1,\bar{v}_2)\in \Omega_{p,f}$ modulo $p^2$. Then $F=pk$ for some $k\in\Z$ coprime to $p$ and the fibre at $(\b{u},\b{v})$ is $X_{\b{u},\b{v}}: kpx^2+Gy^2=z^2$. The reduction  fibre $\Bar{X}_{\b{u},\b{v}}$ to $\F_p$ then takes the form
    \[
    Gy^2\equiv z^2\bmod{p},
    \]
    which is insoluble over $\F_p$ since $G\neq\square\bmod{p}$. It follows that $X_{\b{u},\b{v}}(\Q_p) = \emptyset$. The result follows by a symmetric argument if $(\b{u},\b{v})$ reduces to an element of $\Omega_{p,g}$, thus we may now assume that $(\b{u},\b{v})$ reduces to some point in $\Omega_{p,f,g}$. In this case $F=pk$ and $G=pl$ for $k,l\in\Z$ each coprime to $p$. The fibre in question is therefore given by $X_{\b{u},\b{v}}: pkx^2+ply^2=z^2$. This fibre is equivalent over $\Q_p$ to the ternary conic $kx^2+ly^2=pz^2$, whose reduction to $\F_p$ is
    \[
    kx^2 + ly^2 \equiv 0\bmod{p}.
    \]
    This is soluble if and only if $-kl\equiv\square\bmod{p}$, but this does not hold by the assumption that $(\b{u},\b{v})$ reduces to an element of $\Omega_{p,f,g}$ modulo $p^2$. The result follows.
\end{proof}

Note that the reverse direction does not hold. For example, suppose $(\b{u},\b{v})\in \Z_{prim}^2\times\Z_{prim}^2$ is such that $f(u_1,u_2)=p^3 k$ for $k$ coprime to $p$ and $g(v_1,v_2) \not\equiv\square\bmod{p}$. Then $(\b{u},\b{v})$ does not reduce to a point in $\Omega_p$, yet the corresponding fibre does not have a $\Q_p$ point. Suppose $\Omega_{p^2}$ is precisely the collection of these other bad residue classes modulo $p^2$. Then we have the following partial converse to Lemma \ref{badresidueclassesarebad}.

\begin{lemma}\label{partialconversetobadresidues}
Let $p\nmid 2\mathrm{disc}(f)\mathrm{disc}(g)$ be a prime and suppose that $((u_1,u_2),(v_1,v_2))\in\Z^2_{prim}\times\Z^2_{prim}$ is such that $X_{\b{u},\b{v}}(\Q_p)=\emptyset$. Then either $((u_1,u_2),(v_1,v_2))\bmod{p^2}\in\Omega'_p$, $f(u_1,u_2)\equiv0\bmod{p^2}$ or $g(v_1,v_2)\equiv0\bmod{p^2}$. In particular,
\[
\Omega_{p^2} \subseteq \left\{(u_1,u_2,v_1,v_2)\in\left(\Z/p^2\Z\right): \begin{array}{l} (u_1,u_2),(v_1,v_2)\not\equiv 0\bmod{p},\\ f(u_1,u_2)\equiv0\bmod{p^2}\;\text{or}\;g(v_1,v_2)\equiv0\bmod{p^2}\end{array}\right\}.
\]
\end{lemma}

\begin{proof}
    Suppose that $((u_1,u_2),(v_1,v_2))\in\Z^2_{prim}\times\Z^2_{prim}$ is such that $X_{\b{u},\b{v}}(\Q_p)=\emptyset$, $f(u_1,u_2)\not\equiv0\bmod{p^2}$ and $g(v_1,v_2)\not\equiv0\bmod{p^2}$. We must show that $((u_1,u_2),(v_1,v_2))\bmod{p^2}\in\Omega'_p$. As before, write $F=f(u_1,u_2)$ and $G=g(v_1,v_2)$. If $F,G\not\equiv0\bmod{p}$, then it follows from the Chevalley--Warning theorem that $X_{(u,v)}(\Q_p)\neq\emptyset$. Therefore, we assume that $F\equiv0\bmod{p}$ or $G\equiv0\bmod{p}$.\\
    If $F\equiv0\bmod{p}$ and $G\not\equiv0\bmod{p}$, then the reduced fibre takes the form $Gy^2=z^2\bmod{p}$, which has a solution over $\F_p$ if and only if $G\equiv\square\bmod{p}$. Since $X_{\b{u},\b{v}}(\Q_p)=\emptyset$, we must have $G\not\equiv\square\bmod{p}$, which means that $(\b{u},\b{v})\bmod{p^2}\in\Omega_{p,f}$. Similarly, if $G\equiv0\bmod{p}$ and $F\not\equiv0\bmod{p}$, $(\b{u},\b{v})\bmod{p^2}\in\Omega_{p,g}$. Finally, if $F\equiv0\bmod{p}$ and $G\equiv0\bmod{p}$, then, writing $F=pk$, $G=pl$ where $\mathrm{gcd}(p,k)=\mathrm{gcd}(p,l)=1$, we have that the fibre $X_{\b{u},\b{v}}$ is equivalent to the conic
    \[
    kx^2+ly^2 = pz^2,
    \]
    which fails to have a solution in $\Q_p$ if and only if $-kl\not\equiv\square\bmod{p}$. Therefore we must have $(\b{u},\b{v})\bmod{p^2}\in\Omega_{p,f,g}$.
\end{proof}

Our next lemma computes the size of the sets $\Omega_{p,f}$, $\Omega_{p,g}$, $\Omega_{p,f,g}$ and $\Omega_{p^2}$.

\begin{lemma}
    For $p\nmid 2\mathrm{disc}(f)\mathrm{disc}(g)$ we have,
    \begin{align*}
    &|\Omega_{p,f}| = \begin{cases}
    p^3\left(p-1\right)^3\left(p-\left(\frac{\mathrm{disc}(g)}{p}\right)\right)\;\text{if}\; p\;\text{splits $f$}\\
    0\;\text{otherwise},
    \end{cases}\\
    &|\Omega_{p,g}| = \begin{cases}
    p^3\left(p-1\right)^3\left(p-\left(\frac{\mathrm{disc}(f)}{p}\right)\right)\;\text{if}\; p\;\text{splits $g$}\\
    0\;\text{otherwise},
    \end{cases}\\
    &|\Omega_{p,f,g}| = \begin{cases}
    2p^3\left(p-1\right)^3\;\text{if $p$ splits both $f$ and $g$}\\
    0\;\text{otherwise},
    \end{cases}\\
    &|\Omega_{p^2}| = O\left(p^6\right).
    \end{align*}
\end{lemma}

\begin{proof}
    The arguments for $\Omega_{p,f}$ and $\Omega_{p,g}$ are equivalent and so we focus on the former. Looking at the definition we have that $|\Omega_{p,f}|$ is equal to
    \[
    \left|\left\{(u_1,u_2)\hspace{-3pt}\in\hspace{-3pt}(\Z/p^2\Z)^4:\hspace{-5pt}\begin{array}{c}f(u_1,u_2)\equiv 0 \bmod{p}\\f(u_1,u_2)\not\equiv 0 \bmod{p^2}\\ (u_1,u_2)\not\equiv(0,0)\bmod{p}\end{array}\hspace{-5pt}\right\}\right|\cdot\left|\left\{(v_1,v_2)\in(\Z/p^2\Z)^4:\hspace{-5pt}\begin{array}{c} g(v_1,v_2)\not\equiv\square\bmod{p}\\(v_1,v_2)\not\equiv(0,0)\bmod{p}\end{array}\hspace{-5pt}\right\}\right|.
    \]
    Call the first cardinality $F(p)$ and the second $G(p)$. Let us first handle the $F(p)$. Recall that if $p$ does not split $f$ then
    \begin{equation}\label{modpf}
    f(u_1,u_2) \equiv 0\bmod{p}
    \end{equation}
    has no solution, and so $F(p)=0$. If $p$ does split $f$, then there are $2(p-1)$ non-zero solutions $(u_1,u_2)\in(\Z/p\Z)^2$ to \eqref{modpf}. Each of these solutions have $p(p-1)$ lifts to $(\Z/p^2\Z)^2$ such that $f(u_1,u_2)\not\equiv 0\bmod{p^2}$. Thus $F(p)=2p(p-1)^2$.\\
    Turning to $G(p)$ we have,
    \begin{align*}
    G(p)=p^2\sum_{\substack{k\in(\Z/p\Z)\\k\neq\square}} \#\{(v_1,v_2)\in(\Z/p\Z)^2:g(v_1,v_2)\equiv k\bmod{p}\}.
    \end{align*}
    Note that the sum counts solutions modulo $p$ and that the $p^2$ corresponds to the lifts of such solutions to $(\Z/p^2\Z)^2$. By orthogonality,
    \begin{align*}
        G(p) = p\sum_{\substack{k\in(\Z/p\Z)\\k\neq\square}} \sum_{(v_1,v_2)\in(\Z/p\Z)^2}\sum_{a\in(\Z/p\Z)}e\left(\frac{ag(v_1,v_2)}{p}\right)e\left(\frac{-ak}{p}\right)
    \end{align*}
    where we have used the standard notation $e(z)=e^{2\pi iz}$. Since $p\nmid 2\mathrm{disc}(g)$, we may diagonalise $g$ modulo $p$. Therefore, using an invertible change of variables if necessary, we may assume that
    \[
    g(v_1,v_2) = v_1^2 - \mathrm{disc}(g)v_2^2.
    \]
    Then we have,
    \begin{align*}
    G(p) = p\sum_{\substack{k\in(\Z/p\Z)\\k\not\equiv\square\bmod{p}}}\sum_{a\in(\Z/p\Z)}e\left(\frac{-ak}{p}\right) \sum_{v_2\in(\Z/p\Z)}e\left(\frac{-a\mathrm{disc(g)}v_2^2}{p}\right)\sum_{v_1\in(\Z/p\Z)}e\left(\frac{av_1^2}{p}\right),
    \end{align*}
    which we may compute using standard results for Gauss sums to obtain
    \[
    G(p) = \frac{1}{2}p(p-1)^2\left(p-\left(\frac{\mathrm{disc(g)}}{p}\right)\right).
    \]
    Multiplying this with our expressions for $F(p)$ gives the desired result.\\
    We now turn to $|\Omega_{p,f,g}|$. From the definition of $\Omega_{p,f,g}$, it is clear that it will be empty if $p$ does not split either $f$ or $g$. If $p$ splits both $f$ and $g$, then, again from the definition, for any $(u_1,u_2,v_1,v_2)\in\left(\Z/p^2\Z\right)^4$ there exist $k,l\in\{1,\ldots,p-1\}$ such that $-kl\not\equiv\square\bmod{p}$ and
    \[
    f(u_1,u_2)\equiv pk\bmod{p^2}\;\;\text{and}\;\;g(v_1,v_2)\equiv pl\bmod{p^2}.
    \]
    Using this we may write
    \[
    |\Omega_{p,f,g}| = \frac{1}{p^4}\sum_{\substack{k,l\in(\Z/p\Z)\\-kl\not\equiv\square\bmod{p}}}\widetilde{F}(p,k)\widetilde{G}(p,l)
    \]
    where
    \[
    \widetilde{F}(p,k)=\sum_{\substack{(u_1,u_2)\in(\Z/p^2\Z)^2\\(u_1,u_2)\not\equiv(0,0)\bmod{p}}}\sum_{a\in(\Z/p^2\Z)}e\left(\frac{af(u_1,u_2)}{p^2}\right)e\left(\frac{-apk}{p^2}\right)
    \]
    and
    \[
    \widetilde{G}(p,l) = \sum_{\substack{(v_1,v_2)\in(\Z/p^2\Z)^2\\(v_1,v_2)\not\equiv(0,0)\bmod{p}}}\sum_{b\in(\Z/p^2\Z)}e\left(\frac{bg(v_1,v_2)}{p^2}\right)e\left(\frac{-bpl}{p^2}\right).
    \]
    These are symmetric and will contribute the same amount to the sum, so we will restrict our focus to $\widetilde{F}(p,k)$. Similar to our previous reasoning, we may diagonalise $f$ modulo $p^2$ and thus assume that $f(u_1,u_2)=u_1^2-u_2^2$ since $p$ splits $f$ and so $\mathrm{disc}(f)$ is a square modulo $p$. Then, noting that $(u_1,u_2)\equiv(0,0)\bmod{p}$ contribute null to this sum, we may add in those terms at no extra cost. Rearranging then gives,
    \[
    \widetilde{F}(p,k) = \sum_{a\in(\Z/p^2\Z)}e\left(\frac{-ak}{p}\right)\sum_{\substack{(u_1,u_2)\in(\Z/p^2\Z)^2}}e\left(\frac{au_1^2}{p^2}\right)e\left(\frac{-au_2^2}{p^2}\right).
    \]
    Focusing on the sum over $u_1^2$, and writing $u_1 = s+pt$ for $s,t\in\{0,\ldots,p-1\}$, we have
    \begin{align*} \sum_{\substack{u_1\in(\Z/p^2\Z)}}e\left(\frac{au_1^2}{p^2}\right) &= \sum_{\substack{s\in(\Z/p\Z)}}e\left(\frac{as^2}{p^2}\right)\sum_{\substack{t\in(\Z/p\Z)}}e\left(\frac{ast}{p}\right)\\ 
    =&\sum_{\substack{s\in(\Z/p\Z)}}e\left(\frac{as^2}{p^2}\right)\left(\mathds{1}(a\equiv 0\bmod{p})p + \mathds{1}(a\not\equiv 0 \bmod{p})\mathds{1}(s\equiv 0\bmod{p})p\right)\\
    =& \mathds{1}(a\not\equiv 0 \bmod{p})p + \mathds{1}(a\equiv 0\bmod{p})p\sum_{\substack{s\in(\Z/p\Z)}}e\left(\frac{as^2}{p^2}\right)\\
    =& \mathds{1}(a\not\equiv 0 \bmod{p})p + \mathds{1}(a\equiv 0 \bmod{p},a\not\equiv0\bmod{p^2})\varepsilon_p\left(\frac{a/p}{p}\right) p^{3/2}\\ +& \mathds{1}(a\equiv 0\bmod{p^2})p^2,
    \end{align*}
    where $\varepsilon_p = 1$ if $p\equiv 1\bmod{4}$ and $\varepsilon_p = i$ if $p\equiv 3\bmod{4}$. Doing the same computation for the sum over $u_2$ and multiplying with the above (noting that the Legendre symbol occurring in the $u_2$ case will be $\left(\frac{-a/p}{p}\right)$), we obtain
    \begin{align*}
    \widetilde{F}(p,k) &= p^2\sum_{\substack{a\in(\Z/p^2\Z)\\ a\not\equiv 0 \bmod{p}}}e\left(\frac{-ak}{p}\right) + p^3 \sum_{\substack{a\in(\Z/p\Z)\\ a\not\equiv 0 \bmod{p}}}1 + p^4\\
    & = 2p^3(p-1)
    \end{align*}
    We will obtain the same expression for $\widetilde{G}(p,l)$ and so,
    \[
    |\Omega_{p,f,g}| = \frac{1}{p^4}\sum_{\substack{k,l\in(\Z/p\Z)\\-kl\not\equiv\square\bmod{p}}} 4p^6(p-1)^2 = 2p^3(p-1)^3.
    \]
    Finally, let us bound $|\Omega_{p^2}|$. By Lemma \ref{partialconversetobadresidues}, we have
    \begin{align*}
        |\Omega_{p^2}| &\ll \left|\left\{(u_1,u_2,v_1,v_2)\in\left(\Z/p^2\Z\right): \begin{array}{l} (u_1,u_2),(v_1,v_2)\not\equiv 0\bmod{p},\\ f(u_1,u_2)\equiv0\bmod{p^2}\;\text{or}\;g(v_1,v_2)\equiv0\bmod{p^2}\end{array}\right\}\right|\\
        &\ll p^4\left|\left\{(u_1,u_2)\in\left(\Z/p^2\Z\right)^2: \begin{array}{l} (u_1,u_2)\not\equiv 0\bmod{p},\\ f(u_1,u_2)\equiv0\bmod{p^2}\end{array}\right\}\right|\\ &+ p^4\left|\left\{(v_1,v_2)\in\left(\Z/p^2\Z\right)^2: \begin{array}{l} (v_1,v_2)\not\equiv 0\bmod{p},\\ g(v_1,v_2)\equiv0\bmod{p^2}\end{array}\right\}\right|.
    \end{align*}
    Each of the two cadinalities in the last bound may be computed in the same way, so we only deal with the first. Call it $H(p)$. Using orthogonality and the fact that we may assume $f(u_1,u_2)\equiv u_1^2 -Du_2^2\bmod{p^2}$ for $D=\mathrm{disc}(f)$, we may write,
    \begin{align*}
    p^2 H(p) &= \sum_{a\in(\Z/p^2\Z)}\sum_{\substack{u_1,u_2\in(\Z/p^2\Z)\\ (u_1,u_2)\not\equiv(0,0)\bmod{p}}}e\left(\frac{a(u_1^2-Du_2^2)}{p^2}\right)\\
    &= \sum_{a\in(\Z/p^2\Z)}\sum_{\substack{k,l\in(\Z/p\Z)}}e\left(\frac{ak^2}{p^2}\right)e\left(\frac{akl}{p}\right)\sum_{\substack{m,n\in(\Z/p\Z)\\ (k,m)\not\equiv(0,0)\bmod{p}}}e\left(\frac{-aDm^2}{p^2}\right)e\left(\frac{-aDmn}{p}\right),
    \end{align*}
    the last equality coming from setting $u_1\equiv k+pl\bmod{p^2}$ and $u_2\equiv m+pn\bmod{p^2}$ for $k,l,n,m\in(\Z/p\Z)$. Summing over $n$ and $l$ first we obtain:
    \begin{align*}
        p^2 H(p) &= p^2\sum_{a\in(\Z/p^2\Z)}\sum_{\substack{k,m\in(\Z/p\Z)\\(k,m)\not\equiv(0,0)\bmod{p})}}e\left(\frac{ak^2}{p^2}\right)e\left(\frac{-aDm^2}{p^2}\right)\mathds{1}(am\equiv0\bmod{p})\mathds{1}(ak\equiv0\bmod{p})\\
        &= p^2\sum_{b\in(\Z/p\Z)}\sum_{\substack{k,m\in(\Z/p\Z)\\(k,m)\not\equiv(0,0)\bmod{p})}}e\left(\frac{bk^2}{p}\right)e\left(\frac{-bDm^2}{p}\right),
    \end{align*}
    the second line following by writing $a\equiv bp\bmod{p^2}$ for $b\in(\Z/p\Z)$ using the fact that the conditions $\mathds{1}(am\equiv0\bmod{p})$, $\mathds{1}(ak\equiv0\bmod{p})$ and $(k,m)\not\equiv(0,0)\bmod{p}$ imply that $a$ must be congruent to $0$ modulo $p$. Next we compute the sum over $m$:
    \begin{align*}
    \sum_{\substack{m\in(\Z/p\Z)\\(k,m)\not\equiv(0,0)\bmod{p})}}e\left(\frac{-bDm^2}{p}\right) &= \left(\varepsilon_p \left(\frac{-bD}{p}\right) p^{1/2} - \mathds{1}(k\equiv 0\bmod{p})\right)\mathds{1}(b\not\equiv 0\bmod{p})\\ &+ (p-\mathds{1}(k\equiv 0\bmod{p}))\mathds{1}(b\equiv0\bmod{p}).
    \end{align*}
    Summing this over $k$ then gives
    \begin{align*}
     \left(\left(\frac{D}{p}\right) p - 1\right)\mathds{1}(b\not\equiv 0\bmod{p}) + (p^2-1)\mathds{1}(b\equiv0\bmod{p}).
    \end{align*}
    Thus we obtain,
    \begin{align*}
    p^2 H(p) &= p^2\left(\left(\frac{D}{p}\right)p-1\right)(p-1) + p^2(p^2-1)\\
    &=\left(1+\left(\frac{D}{p}\right)\right)p^4 - \left(1+\left(\frac{D}{p}\right)\right)p^3 \ll p^4.
    \end{align*}
    It follows that $|\Omega_{p^2}| \ll \frac{p^4\cdot p^4}{p^2}\ll p^6$ as required.
\end{proof}

Define $\eta_p$ as follows: set $\eta_p = 0$ if $p| 2\mathrm{disc}(f)\mathrm{disc}(g)$ and if $p\nmid 2\mathrm{disc}(f)\mathrm{disc}(g)$ set
\[
\eta_p = \begin{cases}
    2\;\text{if $p$ splits both $f$ and $g$},\\
    1\;\text{if $p$ splits exactly one of $f$ and $g$},\\
    0\;\text{if $p$ splits neither $f$ nor $g$}.
\end{cases}
\]

We sum up the previous result in the following corollary.

\begin{corollary}\label{keylocallemma}
    For any prime $p$, define $\Omega_p = \Omega'_p \sqcup \Omega_{p^2}$. Then,
    \[
    |\Omega_p| = \eta_p p^7 + O(p^6).
    \]
\end{corollary}

Finally, let $\delta_1(f,g)$, $\delta_2(f,g)$ and $\delta_3(f,g)$ denote the density of primes which split only $f$, only $g$, and both $f$ and $g$ respectively. The following result, which is a direct consequence of the Chebotarev density theorem, gives us the necessary information on these quantities:

\begin{lemma}\label{Chebdense}
    Let $f,g\in\Z[u_1,u_2]$ be binary quadratic forms which do not split as a double line. Then
    \[
    \delta_1(f,g) = \begin{cases}
        0\;\text{if $f$ splits over $\Q$ or if $f$ and $g$ split over the same field},\\
        1/2\;\text{if only $g$ splits over $\Q$},\\
        1/4\;\text{if $f$ and $g$ split over the same quadratic extension of $\Q$},
    \end{cases}
    \]
    \[
    \delta_2(f,g) = \begin{cases}
        0\;\text{if $g$ splits over $\Q$ or if $f$ and $g$ split over the same field},\\
        1/2\;\text{if only $f$ splits over $\Q$},\\
        1/4\;\text{if $f$ and $g$ split over different quadratic extensions of $\Q$},
    \end{cases}
    \]
    \[
    \delta_3(f,g) = \begin{cases}
        1\;\text{if both $f$ and $g$ split over $\Q$},\\
        1/2\;\text{if $f$ and $g$ split over the same quadratic extension, or if only one splits over $\Q$},\\
        1/4\;\text{if $f$ and $g$ split over different quadratic extensions of $\Q$}.
    \end{cases}
    \]
\end{lemma}

\section{Application of the large sieve}
In this section we will apply the large sieve to prove Theorem \ref{acheightbound}.

\subsection{Preliminaries}
We will use the following version of the large sieve:

\begin{lemma}\label{LARGEsieve} Let
\[
X = \{\b{n}\in\Z^r: M_j\leq n_j\leq M_j + N_j,\;\text{for}\;1\leq j\leq r\}
\]
and suppose $Y_p \subseteq (\Z/p^s\Z)^r$ for all $p$ where $s\in\N$ is fixed. Then
\[
\{\b{n}\in X: \b{n}\bmod{p^s}\not\in Y_p,\;\text{for all}\; p\leq L\} \ll \frac{\prod_{i=1}^{r}\left(\sqrt{N_j}+L^{s}\right)^2}{F(L)}
\]
where
\[
F(L) = \sum_{m\leq L}\mu^2(m)\prod_{p|m}\left(\frac{|Y_p|}{p^{rs}-|Y_p|}\right).
\]
\end{lemma}

\begin{proof}
    This particular version of the large sieve is slightly different from versions used in similar problems - in \cite{BLS} they use a version with skewed boxes but where $Y_p$ is only a subset of $\Z/p\Z$, while the standard version of the sieve used for such problems, given in \S6 of \cite{SerreConics}, does not allow for skewed boxes. Both of these, and that given in Lemma \ref{LARGEsieve}, may be seen to be an application of Theorem $1$ of \cite{Huxely} followed by Proposition $2.3$ of \cite{KowalskiLargeSieve}.
\end{proof}

To prove Theorems \ref{acheightbound} we will apply Lemma \ref{LARGEsieve} with $Y_p = \Omega_p$ as given in \S\ref{Localdensities}, so $s=2$ and $r=4$. The following lemma computes lower bounds for the expression $F(L)$:

\begin{lemma}\label{Savinglemma}
    For notation as above, we have
    \[
    \sum_{m\leq L}\mu^2(m)\prod_{p|m}\left(\frac{|\Omega_p|}{p^{8}-|\Omega_p|}\right) \sim c(\log L)^{\Delta(\pi)}
    \]
    where $\Delta(\pi)$ is given in Proposition \ref{DELTA} and $c>0$ is some constant.
\end{lemma}

\begin{proof}
    Define the multiplicative function $h(m) = \mu^2(m)\prod_{p|m}\left(\frac{|\Omega_p|}{p^{8}-|\Omega_p|}\right)$. We aim to compute
    \[
    \sum_{m\leq L} h(m)
    \]
    for which we will apply Theorem $A.5$ from \cite{Cribro}. Upon appealing to Corollary \ref{keylocallemma}, we have
    \begin{align*}
    \sum_{p\leq x} h(p)\log p &= \sum_{p\leq x} \frac{\eta_p\log p}{p} + O\left(\sum_{p\leq x} \frac{\log p}{p^2}\right)\\
    &= \sum_{\substack{p\leq x\\ p\;\text{only splits $f$}}}\hspace{-5pt}\frac{\eta_p\log p}{p} + \sum_{\substack{p\leq x\\ p\;\text{only splits $g$}}}\hspace{-5pt}\frac{\eta_p\log p}{p} + \sum_{\substack{p\leq x\\ p\;\text{splits both $f$ and $g$}}}\hspace{-7.5pt}\frac{\eta_p\log p}{p} + O(1)\\
    &= \left(\sum_{i=1}^{3}\widetilde{\eta}_i\delta_i(f,g)\right) \log x + O(1)
    \end{align*}
    where
    \begin{equation}\label{etatildedef}
    \widetilde{\eta}_i = \begin{cases}
        1\; \text{if}\; i=1,2\\
        2\;\text{if}\; i=3.
    \end{cases}
    \end{equation}
    This follows from the Chebotarev density theorem and Lemma \ref{Chebdense} and gives condition $(A.15)$ of Theorem $A.5$ of \cite{Cribro}. We may similarly prove condition $(A.16)$ and we clearly have,
    \[
    \sum_{p}h(p)^2\log p \ll \sum_{p} \frac{\log p}{p^2} < \infty,
    \]
    which gives condition $(A.17)$. Thus Theorem $A.5$ of \cite{Cribro} implies that
    \[
    \sum_{m\leq L} h(m) = c_h(\log L)^{k} + O((\log L)^{k-1})
    \]
    for some constant $c_h>0$ where
    \[
    k = \sum_{i=1}^{3}\widetilde{\eta}_i\delta_i(f,g).
    \]
    To conclude we note that by \eqref{etatildedef} and Lemma \ref{Chebdense}, $k=\Delta(\pi)$ regardless of the splitting behaviour of $f$ and $g$.
\end{proof}

\subsection{Proof of Theorem \ref{acheightbound}}
We will follow the method of \cite{BLS}. The key difference is that, due to the definition of our residue sets, we do not need to reduce to square-free variables. Define
\[
N^{*}(T_1,T_2,S_1,S_2) = \left\{(\b{u},\b{v})\in\Z^2_{prim}\times\Z^2_{prim}: \begin{array}{l} T_i\leq u_i\leq 2T_i,\; S_i\leq v_i\leq 2S_i\;\forall i\in\{0,1\}\\ (\b{u},\b{v})\bmod{p^2}\not\in\Omega_p\;\forall p\leq R^{1/100} \end{array} \right\}.
\]
where $R\coloneqq \min(T_1,T_2,S_1,S_2)$. Then we have
\begin{equation}\label{diadicsplitting}
N_{f,g}(B) \ll \sum_{T_1,T_2,S_1,S_2} N^{*}(T_1,T_2,S_1,S_2)
\end{equation}
where the sum is over those $T_1,T_2,S_1,S_2$ which are powers of $2$ and satisfy the height condition
\[
\max(T_1,T_2)\cdot\max(S_1,S_2)\leq \sqrt{B}.
\]
Now we may apply Lemma \ref{LARGEsieve} and Lemma \ref{Savinglemma} to each $N^{*}$. This will yield,
\[
N^{*}(T_1,T_2,S_1,S_2) \ll \min\left(T_1T_2S_1S_2, \frac{T_1T_2S_1S_2}{(\log R)^{\Delta(\pi)}} \right) \ll \frac{T_1T_2S_1S_2}{1 + (\log R)^{\Delta(\pi)}}.
\]
By symmetry it is enough to sum \eqref{diadicsplitting} over the region $T_1\leq T_2$, $S_1\leq S_2$. Summing over $T_2\leq 4 \sqrt{B}/S_2$ we obtain
\[
N_{f,g}(B) \ll \sqrt{B}\sum_{T_1,S_1,S_2\leq \sqrt{B}}\frac{T_1S_1}{1+(\log R)^{\Delta(\pi)}}.
\]
Now, if $T_1\leq S_1$, we first sum over $S_1\leq S_2$ and then $S_2\leq 4\sqrt{B}/T_1$ to obtain
\[
N_{f,g}(B) \ll B\sum_{T_1\leq \sqrt{B}}\frac{1}{1+(\log T_1)^{\Delta(\pi)}} \ll B\sum_{j\leq \log_2 \sqrt{B}} \frac{1}{1+j^{\Delta(\pi)}}.
\]
Performing this sum over $j$ will give the desired bounds upon noting that $\Delta(\pi) \geq 3/2$ if at least one of $f$ and $g$ split over $\Q$ and $\Delta(\pi)=1$ if neither split over $\Q$. It remains to consider the region where $S_1\leq T_1$. Here we sum over $T_1\leq 4\sqrt{B}/S_2$ and then $S_2\geq S_1$ to obtain
\[
N_{f,g}(B) \ll B\sum_{S_1\leq \sqrt{B}} \frac{1}{1+(\log S_1)^{\Delta(\pi)}} \ll B\sum_{j\leq \log_2 \sqrt{B}}\frac{1}{1+j^{\Delta(\pi)}}
\]
which again gives the desired bound.

\section{Bounding of thin sets}
Finally we will prove the bound given in Theorem \ref{thinsetbound}. Part $(1)$ is trivial. For part $(2)$, suppose $U\in\P^{1}(\Q)$ is such that $f(u)=n^2$ for some $n\in\Z$. Then for any $v\in\P^{1}(\Q)$, with $H((U,v))\leq B$, $(U,v)\in N_{\textrm{thin},1}(B)$. There are $\gg B$ such $v\in\P^1(\Q)$ for fixed $U$, giving the lower bound. The upper bound is implied by Theorem \ref{acheightbound} if at least one of $f$ or $g$ splits over $\Q$, but in the case that neither of them do we want to prove a tighter bound. It is enough to bound
\[
N_{1,thin}^{f}(B) = \left\{(\b{u},\b{v})\in\Z_{prim}^{2}\times\Z_{prim}^{2}: f(u_1,u_2)=\square,\;\max\{|u_1|,|u_2|\}^2\cdot\max\{|v_1|,|v_2|\}^2\leq B\right\}
\]
and
\[
N_{1,thin}^{g}(B) = \left\{(\b{u},\b{v})\in\Z_{prim}^{2}\times\Z_{prim}^{2}: g(v_1,v_2)=\square,\;\max\{|u_1|,|u_2|\}^2\cdot\max\{|v_1|,|v_2|\}^2\leq B\right\}.
\]
The same argument will bound both, so let us focus on $N^{f}_{1,thin}(B)$. We have
\[
N^{f}_{1,thin}(B) = \sum_{\substack{(v_1,v_2)\in\Z_{prim}^2\\ |v_1|,|v_2|\leq B^{1/2}}} \#\left\{(u_1,u_2,y)\in\Z_{prim}^3: \begin{array}{l} f(u_1,u_2)=y^2,\\|y|^2\cdot\max\{|v_1|,|v_2|\}^2\leq 3\|f\|B\\ \max\{|u_1|,|u_2|\}^2\cdot\max\{|v_1|,|v_2|\}^2\leq B \end{array}\right\}.
\]
Here we use $\|f\|$ to denote the maximum of the absolute value of the coefficients. We will bound the inner quantities using the following result:
\begin{lemma}\label{BHBbound}
    Let $q(x_1,x_2,x_3)$ be a non-singular integral ternary quadratic form with matrix $\b{M}$. Set $D=\det{\b{M}}$ and set $D_0$ to be the highest common factor of the $2\times 2$ minors of $\b{M}$. Then there are $$O\left(1+\left(\frac{L_1L_2L_3(D_0)^{3/2}}{D}\right)^{1/3}\right)\tau(D)$$ primitive integer solutions to $q(x_1,x_2,x_3)=0$ satisfying $|x_i|\leq L_i$ for $1\leq i\leq 3$.
\end{lemma}
This is the three variable case of \cite[Corollary 2]{BH-B3}. Here we have $D,D_0$ and $\tau(D)$ are $O(1)$. Thus we have,
\[
N^{f}_{1,thin}(B) \ll \sum_{\substack{(v_1,v_2)\in\N^2\\ v_1,v_2\leq B^{1/2}}} \left(1+\left(\frac{B^{3/2}}{\max\{|v_1|,|v_2|\}^3}\right)^{1/3}\right) \ll B + B^{1/2}\sum_{\substack{(v_1,v_2)\in\N\\ v_1,v_2\leq B^{1/2}}} \frac{1}{v_1^{1/2}v_2^{1/2}}\ll B.
\]

For $(3)$ we first note that we have $f(u_1,u_2)g(v_1,v_2)=\square$ if and only if there exist integers $z,y_1,y_2\in\Z^2$ such that $z$ is square-free and $f(u_1,u_2)=zy_1^2$ and $g(v_1,v_2) = zy_2^2$. We can furthermore assert that
\[
zy_1^2 \ll \|(u_1,u_2)\|^2\;\text{and}\;zy_2^2\ll\|(v_1,v_2)\|^2.
\]
Thus, if $H_2((u,v))\leq B$, we will have
$z^2y_1^2y_2^2 \leq B$. Now, by fixing $z$, we may write,
\[
N_{2,\textrm{thin}}(B) \ll \hspace{-5pt} \sum_{\substack{z\ll B^{1/2}\\ z\;\textrm{sqrf}}} \hspace{-5pt}\#\left\{(u_1,u_2,y_1),(v_1,v_2,y_2)\in\Z_{prim}^3: \begin{array}{l} f(u_1,u_2) = zy_1^2,\;g(v_1,v_2) = zy_2^2,\\
\max\{|u_1|,|u_2|\}^2\cdot\max\{|v_1|,|v_2|\}^2\leq B,\\ |y_1|^2\cdot|y_2|^2 \leq B/z^2 
\end{array}\right\}.
\]
The fact that we are counting over points in $\Z_{prim}^3$ follows from the fact that $(u_1,u_2)\in\Z_{prim}^2$ and $(v_1,v_2)\in\Z_{prim}^2$, since these pairs represent points in $\P^1(\Q)$. We once again use diadic boxes to handle the height conditions: let $T,S,R_1,R_2$ run over powers of $2$ such that
\begin{equation}\label{diadicheightconditions2}
TS \leq B^{1/2}\;\text{and}\; R_1R_2\leq B^{1/2}/z.
\end{equation}
Then define
\[
N^{f}_{2,\textrm{thin}}(T,R_1) = \#\left\{(u_1,u_2,y_1)\in\Z_{prim}^3: \begin{array}{l} f(u_1,u_2)=zy_1^2,\\
T<\max\{|u_1|,|u_2|\}\leq 2T,\\
R_1<|y_1|\leq 2R_1\end{array}\right\}
\]
and
\[
N^{g}_{2,\textrm{thin}}(S,R_2) = \#\left\{(v_1,v_2,y_2)\in\Z_{prim}^3: \begin{array}{l} g(v_1,v_2)=zy_2^2,\\
S<\max\{|v_1|,|v_2|\}\leq 2S,\\
R_2<|y_2|\leq 2R_2\end{array}\right\}.
\]
Thus we may write,
\begin{equation}\label{diadicsplittingcheckpoint1}
    N_{2,\textrm{thin}}(B) \ll \sum_{\substack{z\ll B^{1/2}\\ z\;\textrm{sqrf}}} \sum_{\substack{T,S\\R_1,R_2\\\eqref{diadicheightconditions2}}} N_{2,\textrm{thin}}^{f}(T,R_1) N_{2,\textrm{thin}}^{g}(S,R_2).
\end{equation}
We now note that $N_{\textrm{thin}}^{f}(T,R_1)$ and $N_{\textrm{thin}}^{g}(S,R_2)$ count primitive solutions to the integral ternary quadratic forms
\begin{equation}\label{ternaryquadraticformsforthinsets}
f(u_1,u_2) = zy_1^2\;\text{and}\;g(v_1,v_2) = zy_2^2
\end{equation}
within the relevant boxes. Furthermore, each of these ternary quadratic forms is non-singular, since $f(u_1,u_2)$ and $g(v_1,v_2)$ do not split as double lines. We may therefore bound these quantities by applying Lemma \ref{BHBbound}. For the quadratic forms in \eqref{ternaryquadraticformsforthinsets} we have,
\[
D_0 \ll 1,\; z\ll D\ll z,\;\text{and}\;\tau(z)\ll\tau(D)\ll\tau(z).
\]
where the constants may depend on $f$ or $g$ as appropriate. From this it follows that, for any $\epsilon>0$,
\[
N_{\textrm{thin}}^{f}(T,R_1) \ll \left(1+\left(\frac{T^2R_1}{z}\right)^{1/3}\right)\tau(z)\;\text{and}\;N_{\textrm{thin}}^{g}(S,R_2) \ll \left(1+\left(\frac{S^2R_2}{z}\right)^{1/3}\right)\tau(z).
\]
Thus,
\begin{align*}
N_{\textrm{thin}}(B) &\ll \sum_{\substack{z\ll B^{1/2}\\ z\;\textrm{sqrf}}} \tau(z)^2\sum_{\substack{T,S\\R_1,R_2\\\eqref{diadicheightconditions2}}} \left(1+\left(\frac{T^2R_1}{z}\right)^{1/3}+\left(\frac{S^2R_2}{z}\right)^{1/3}+\left(\frac{T^2S^2R_1R_2}{z^2}\right)^{1/3}\right)\\
&\ll \sum_{\substack{z\ll B^{1/2}\\ z\;\textrm{sqrf}}} \tau(z)^2\sum_{\substack{T,S\\R_1,R_2\\\eqref{diadicheightconditions2}}} \left(1+\left(\frac{T^2S^2R_1R_2}{z}\right)^{1/3}\right).
\end{align*}
We now use the conditions \eqref{diadicheightconditions2} and the fact that there are $O(\log B)$ possible values for each of the $T,S,R_1$ and $R_2$. Then,
\[
N_{\textrm{thin}}(B) \ll (\log B)^6\sum_{\substack{z\ll B^{1/2}\\ z\;\textrm{sqrf}}} \tau(z)^2\left(1+\frac{B^{1/2}}{z^{2/3}}\right) \ll (B^{1/2}+B^{2/3})(\log B)^9.
\]

\end{document}